\documentclass{amsart}
\usepackage{amssymb}
\newtheorem{theorem}{Theorem}[section]

\theoremstyle{definition}

\newtheorem{prp}[equation]{Proposition}
\newtheorem{cor}[equation]{Corollary}

\newtheorem{conj}[equation]{Conjecture}

\theoremstyle{remark}
\newtheorem{remark}[theorem]{Remark}

\numberwithin{equation}{section}

\begin{document}

\title[Ergodicity of Falling Particle Systems]{Proof of Wojtkowski's Falling Particle Conjecture}


\author{Nandor Simanyi}
\address{1402 10th Avenue South \\
Birmingham AL 35294-1241}
\email{simanyi@uab.edu}



\date{\today}

\begin{abstract}

  In this paper we present an unconditional proof of Wojtkowski's
  Ergodicity Conjecture for almost every system of $1D$ perfectly
  elastic balls falling down in a half line under constant
  gravitational acceleration, \cite{W1985}, \cite{W1986},
  \cite{W1990a}, \cite{W1990b}, \cite{W1998}. Namely, by introducing a
  new algebraic approach, we prove that almost every such system is
  (completely hyperbolic and) ergodic.
  
\end{abstract}

\maketitle

\section{Introduction/Prerequisites}

In order to introduce the subject of our investigation, the system of
$1D$ falling balls subjected to constant gravitation, along with the
employed technicalities, we will be closely following the first two
sections of \cite{S2024}. In order to make this presentation
self-contained and easier to read, we quote below two passages of
those two sections of \cite{S2024}, essentially verbatim.

In his paper \cite{W1990a} M. Wojtkowski introduced the following
Hamiltonian dynamical system with discontinuities: There is a vertical
half line $\left\{q|\, q\ge 0\right\}$ given and $n$ ($\ge 2$) point
particles with masses $m_1\ge m_2\ge \dots\ge m_n>0$ and positions
$0\le q_1\le q_2\le\dots\le q_n$ are moving on this half line so that
they are subjected to a constant gravitational acceleration $a=-1$
(they fall down), they collide elastically with each other, and the
first (lowest) particle also collides elastically with the hard floor
$q=0$. We fix the total energy

\begin{equation*}
H=\sum_{i=1}^n \left(m_iq_i+\frac{1}{2}m_i v_i^2 \right)
\end{equation*}

by taking $H=1$. The arising Hamiltonian flow with collisions
$\left(\mathbf{M}, \{\psi^t \}, \mu\right)$ ($\mu$ is the Liouville measure) is the 
studied model of this paper. 

Before formulating the result of this article, however, it is worth mentioning
here three important facts:

\begin{enumerate}

\item[$(1)$] Since the phase space $\mathbf{M}$ is compact, the Liouville measure $\mu$
is finite.

\item[$(2)$] The phase points $x\in\mathbf{M}$ for which the trajectory $\{\psi^t(x)|, t \in\mathbb{R}\}$ hits at least
one singularity (i. e. a multiple collision) are contained in a countable 
union of proper, smooth submanifolds of $\mathbf{M}$ and, therefore, such points
form a set of $\mu$ measure zero.

\item[$(3)$] For $\mu$-almost every phase point $x\in\mathbf{M}$ the collision times of
  the trajectory $\{\psi^t(x)|, t \in\mathbb{R}\}$ do not have any finite accumulation point, see
  Proposition A.1 of \cite{S1996}.

\end{enumerate}

In the paper \cite{W1990a} Wojtkowski formulated his main conjecture pertaining
to the dynamical system $\left(\mathbf{M}, \{\psi^t \}, \mu\right)$:

\begin{conj}[Wojtkowski's Conjecture]
  If $m_1\ge m_2\ge\dots\ge m_n>0$ and
  $m_1\ne m_n$, then all but one characteristic (Lyapunov) exponents of the
  flow $\left(\mathbf{M}, \{\psi^t \}, \mu\right)$ are nonzero. Futhermore,
  the system is ergodic.
\end{conj}

\begin{remark}

1. The only exceptional exponent zero must correspond to the flow direction.

\medskip

2. The condition of nonincreasing masses (as above) is essential for 
establishing the invariance of the symplectic cone field --- an important
condition for obtaining nonzero characteristic exponents. As Wojtkowski
pointed out in Proposition 4 of \cite{W1990a}, if $n=2$ and $m_1<m_2$, then there
exists a linearly stable periodic orbit, thus dimming the chances of proving
ergodicity.
\end{remark}

\medskip

In the paper \cite{S2024} we proved Wojtkowski's Ergodicity
Conjecture 1.1 for almost every selection of masses $m_1>m_2> \dots
>m_n$, provided that the Transversality Conditions (Claim 3.1 of
\cite{S2024}) is verified, i. e. singularities of different order
are transversal to each other and, analogously, the stable and
unstable local invariant manifolds are transversal to all
singularities.

Here our main result is to prove the above Transversality Conditions
and, as the main corollary, we obtain our

\begin{theorem}[Main Theorem]
  For almost every selection of masses $m_1>m_2> \dots >m_n$ the falling ball flow
  $(M, \{\psi^t\}, \mu)$ is (completely hyperbolic and) ergodic.
\end{theorem}

\medskip

We recall that the corresponding billiard map (Poincar\'e section) $(\partial M, T, \nu)$
is an invertible dynamical system $T$ mapping the boundary

\[
\partial M=\left\{(q,v)\in M \big|\; q\in\partial Q\right\}
\]

of the phase space $M$ onto itself and preserving the finite measure
$\nu$ on $\partial M$ that can be obtained by projecting along the
flow the invariant measure $\mu$ of the flow onto $\partial M$.
Also, as usual, in $\partial M$ one identifies the pre-collision phase point
$(q, v^-)\in\partial M$ with the post-collision phase point $(q, v^+)\in\partial M$.

Finally, denote by
\[
\Pi:\, \partial M \to \partial Q
\]
the natural projection of $\partial M$ onto $\partial Q$, i. e. $\Pi(q,v)=q$.

\bigskip

\section{The Geometry of the Symplectic Flow}

We will be working with the symplectic coordinates $(\delta h, \delta
v)$ for the tangent vectors of the reduced phase space $\mathbf{M}$
satisfying the usual reduction equations $\sum_{i=1}^n \delta
h_i=0=\sum_{i=1}^n \delta v_i$.

\begin{remark}
  The coordinates $\delta h_i$ and $\delta v_i$ ($i=1,2,\dots,n$)
  serve as suitable symplectic coordinates in the codimension-one
  subspace $\mathcal{T}_x$ of the full tangent space
  $\mathcal{T}_x\mathbf{M}$ of $\mathbf{M}$ at $x$. Recall that the
  $(2n-2)$-dimensional vector space $\mathcal{T}_x$ is transversal to
  the flow direction, and the restriction of the canonical symplectic
  form

  \begin{equation*}
  \omega=\sum_{i=1}^n \delta q_i \wedge \delta p_i=\sum_{i=1}^n \delta
  h_i \wedge \delta v_i
  \]
  of $\mathbf{M}$ is non-degenerate on $\mathcal{T}_x$, see \cite{W1990a}. We also recall that
  \[
  \delta h_i=m_i\delta q_i+m_iv_i\delta v_i=m_i\delta q_i+v_i\delta p_i.
  \end{equation*}

\end{remark}

Corresponding to the above choice of symplectic coordinates, the
considered monotone Q-form will be

\begin{equation}
Q_1(\delta h, \delta v)=\sum_{i=1}^n \delta h_i\delta v_i.
\end{equation}

It is clear that the evolution of $DS^t(\delta h(0), \delta
v(0))=(\delta h(t), \delta v(t))$ between colisions is

\begin{equation}\label{evolution}
\frac{d}{dt}\left(\delta h(t), \delta v(t)\right)=(0, 0).
\end{equation}

If a collision of type $(i, i+1)$ ($i=1,2,\dots,n-1$) takes place at
time $t_k$, then the derivative of the flow at the collision $\delta
h^-(t_k)\mapsto \delta h^+(t_k)$, $\delta v^-(t_k)\mapsto \delta
v^+(t_k)$ is given by the matrices

\begin{equation}\label{nonfloorcollision}
\begin{aligned}
  & \delta h^+(t_k)=R^*_i\left[\delta h^-(t_k)+S_i\delta v^-(t_k)\right] \\
  & \delta v^+(t_k)=R_i\delta v^-(t_k),
\end{aligned}
\end{equation}

where the matrix $R_i$ is the $n\times n$ identity matrix, except that
its $2\times 2$ submatrix at the crossings of the $i$-th and
$(i+1)$-st rows and columns is

\begin{equation*}
  R_i^{(i,i+1)} = \left[
    \begin{array}{cc}
      \gamma_i & 1-\gamma_i \\
      1+\gamma_i & -\gamma_i
    \end{array}
\right]
\end{equation*}

with $\gamma_i=\dfrac{m_i-m_{i+1}}{m_i+m_{i+1}}$. The matrix $S_i$ is,
similarly, the $n\times n$ zero matrix, except its $2\times 2$
submatrix at the crossings of the $i$-th and $(i+1)$-st rows and
columns, which takes the form of

\begin{equation*}
S_i^{(i, i+1)}=
\left[
    \begin{array}{cc}
      \alpha_i & -\alpha_i \\
      -\alpha_i & \alpha_i
    \end{array}
\right]
\end{equation*}

with

\begin{equation}\label{alpha}
  \alpha_i=\frac{2m_im_{i+1}(m_i-m_{i+1})}{(m_i+m_{i+1})^2}(v_i^- - v_{i+1}^-)>0.
\end{equation}

These formulas can be found, for example, in Sention 4 of \cite{W1990a}.

Concerning a floor collision $(0, 1)$ at time $t_k$, the
transformations are

\begin{equation}\label{floorcollision}
\begin{aligned}
  & \delta h_1^+(t_k)=\delta h_1^-(t_k) \\
  & \delta v_1^+(t_k)=\delta v_1^-(t_k) + \frac{2\delta h_1^-(t_k)}{m_1 v_1^+(t_k)},
\end{aligned}
\end{equation}

see, for instance, Section 4 of \cite{W1990a} or \cite{W1998}.

\bigskip

\section{Proof of the Transversality Conditions}

If one closely studies the ergodicity proofs based upon the Birkhoff-Sinai Zig-zag Method,
like the one in \cite{L-W1995}, one realizes that, whenever a property is needed to be proved
for singular phase points $x_0=(q_0,v_0)\in\mathcal{S}_0$ (like the transversality of $\mathcal{S}_k$
to $\mathcal{S}_0$ at $x_0$, or the transversality of the local stable manifold $\gamma^s(x_0)$ to
$\mathcal{S}_0$), it is always enough to estabslish the required property for almost every point
$x_0$ of $\mathcal{S}_0$ with respect to the hypersurface measure of $\mathcal{S}_0$.
This is what we do in this section.

\medskip

We will be focusing on the billiard map $(\partial M, T, \nu)$. Our first result asserts that
the system $(\partial M, T, \nu)$ has no focal points almost surely almost everywhere.

\begin{prp}

  For almost every selection of the masses $m_1>m_2>\dots>m_n$ it is true that for almost every phase
  point $(q_0, v_0)\in \partial M$, for every positive integer $k$, and for every small enough
  $\epsilon>0$ the map

  \begin{equation}\label{candle_maps_onto}
    \Pi\circ T^k:\; C_{\epsilon}(x_0) \to \partial Q
  \end{equation}

  is of full rank (i. e. locally onto) at $x_0$, where

  \begin{equation}\label{candle}
    C_{\epsilon}(x_0)=\left\{(q_0, v)\in \partial M\big|\; \Vert v-v_0\Vert<\epsilon\right\}
  \end{equation}

  is the so called ``candle manifold'', and

  \begin{equation}\label{projection}
    \Pi:\; \partial M\to \partial Q
  \end{equation}

  is the natural projection, taking $\Pi(q,v)=q$ for $(q,v)\in\partial M$.

\end{prp}

\begin{proof}

  It is enough to prove the proposition for a given $k$ and a given symbolic collision sequence
  $\Sigma=(\sigma_0, \sigma_1, \dots, \sigma_k)$, where $\sigma_l(i_l, i_l+1)$ tells that the
  collision $(i_l, i_l+1)$ takes place at $T^l(x_0)$, where $i_l=0$ indicates a floor collision.

  According to $(9)$ of \cite{W1990a}, the time evolution (abrupt change) at a ball-to-ball collision
  $(i, i+1)$ is given by

  \begin{equation}\label{v_collision}
    \begin{aligned}
    & v_i^+=\gamma_i v_i^- + (1-\gamma_i)v_{i+1}^- \\
    & v_{i+1}^+=(1+\gamma_i)v_i^- -\gamma_i v_{i+1}^-,
    \end{aligned}
  \end{equation}

where

\[
\gamma_i=\frac{m_i-m_{i+1}}{m_i+m_{i+1}},
\]

and

\begin{equation}\label{v_1_collision}
  v_1^+ = -v_1^-
  \end{equation}

for any floor collision. It is obvious that the time evolution between collisions is given by

  \begin{equation}\label{time_evolution}
     \frac{d}{dt}q_t = v_t,\; \frac{d}{dt}v_t = (-1,-1,\dots,-1).
  \end{equation}

  This means that the entire trajectory segment $\left\{T^l(x_0)\big|\; l=0, 1, \dots, k\right\}$
  is governed by rational functions of the initial data $(q_0, v_0, \vec{m})$, including the moments
  $t_l$ of the collisions $\sigma_l$.

  Furthermore, by taking derivative of the flow \ref{v_collision}--\ref{time_evolution}, one obtains that for any tangent
  vector

  \[
  \tau_0=(\delta q_0, \delta v_0)\in\mathcal{T}_{x_0}\partial M
  \]

  the images $D\psi^t[\tau_0]=\tau_t$ are evolving in time as follows:

  \begin{equation}\label{delta_v_collision}
    \begin{aligned}
    & \delta v_i^+=\gamma_i \delta v_i^- + (1-\gamma_i)\delta v_{i+1}^- \\
    & \delta v_{i+1}^+=(1+\gamma_i)\delta v_i^- -\gamma_i\delta v_{i+1}^-,
    \end{aligned}
    \end{equation}

  \begin{equation}\label{delta_q_collision}
    \begin{aligned}
    & \delta q_i^+=\gamma_i \delta q_i^- + (1-\gamma_i)\delta q_{i+1}^- \\
    & \delta q_{i+1}^+=(1+\gamma_i)\delta q_i^- -\gamma_i\delta q_{i+1}^-,
    \end{aligned}
    \end{equation}

for $i>0$, where $\delta v_j^{\pm}=(\delta v^{\pm})_{t_j}$, $\delta q_j^{\pm}=(\delta q^{\pm})_{t_j}$, 
for $j=i, i+1$, see (5) in \cite{W1990b}.

At a floor collision $(0,1)$ we clearly have

  \begin{equation}\label{delta_q_floor_collision}
      \delta q_1^+ = -\delta q_1^-,\; \delta v_1^+ = -\delta v_1^-.
  \end{equation}

  It is also clear that the time evolution of the image tangent vector
  $\tau_t=D\psi^t[\tau_0]=(\delta q_t, \delta v_t)$ between collisions is given by

  \begin{equation}\label{delta_time_evolution}
     \frac{d}{dt}\delta q_t = \delta v_t,\; \frac{d}{dt}\delta v_t=0.
  \end{equation}

  It follows from \ref{delta_v_collision}--\ref{delta_time_evolution} that the time evolution
  of the tangent vectors $\tau_t=D\psi^t[\tau_0]=(\delta q_t, \delta v_t)$ is also fully governed
  by finitely many rational functions of the initial data $(q_0, v_0, \vec{m})$.

  Finally, the negation of the assertion of the Proposition for a particular phase point
  $x_0=(q_0, v_0)$ means that the system of homogeneous linear equations

  \begin{equation}\label{homogeneous}
    \delta q_{t_k}=0
  \end{equation}

  has a nontrivial solution $\tau_0=(0, \delta v_0)$. This, in turn, means that cetain minors of
  this system vanish, i.e.

  \begin{equation}\label{vanishing}
    R_j(q_0, v_0, \vec{m})=0, \; j=1,2,\dots,m
  \end{equation}

  for certain rational functions $R_1, R_2,\dots, R_m$ of the initial variables.

  Consider now the limiting system with

  \begin{equation}\label{limiting}
    m_1=m_2=\dots=m_n>0.
  \end{equation}

  This system may not possess a trajectory segment with our prescribed symbolic collision
  sequence $\Sigma=(\sigma_0, \sigma_1,\dots, \sigma_k)$, yet all the rational functions $R_j$
  above are well defined, and they correspond to the time evolution of this limiting system,
  so that the unwanted collisions are annihilated in such a way that any two particles, about to
  making an unwanted collision, are let to penetrate through each other without interaction.
  This limiting system is essentially integrable, after the
  dynamic change of labels at collisions, as if the particles just penetrate through each other
  without interaction, see also the paragraph right after Corollary 2.23 in \cite{S2024}.

  In this limiting case, however, we have that $\delta q_t=t\delta v_0$, after the mentioned dynamic
  change of labels. This means, in turn, that the system of homogeneous linear equations
  \ref{homogeneous} only has the trivial solution $\tau_0=0$, therefore, not all rational functions
  $R_j$ are identically zero, even if we assume $m_1=m_2=\dots=m_n$.

  This completes the proof of Proposition 3.1.

\end{proof}

\medskip

Our next result is analogous to the statement of Proposition 3.1, claimed here for almost every singular phase point
$x_0=(q_0, v_0)\in\mathcal{S}_0$.

  \begin{prp}

    For almost every selection of the masses $m_1>m_2>\dots>m_n$ and for every positive integer $k$
    it is true that for almost every singular phase point $x_0=(q_0, v_0)\in\mathcal{S}_0$ and for all
    small enough $\epsilon>0$ the map in \ref{candle_maps_onto} is of maximum rank (i. e. locally onto)
    at $x_0$.

  \end{prp}

  \begin{proof}

    In Proposition 3.1, $q_0$ possessed $n-1$ independent coordinates. (Keep in mind, that $q_0\in\partial Q$, so the $q$ coordinates of two colliding
    particles are equal, or $q_1=0$.) When a singular collision takes place at time zero, i. e. $x_0=(q_0, v_0)\in\mathcal{S}_0$, then one more equation
    holds true for the configuration coordinates, so one more coordinate should be eliminated to work with independent configuration coordinates.
    The arising vector with $n-2$ coordinates is denoted by $\tilde{q}_0$. After this, the proof of this proposition is verbatim the same as that of
    Proposition 3.1, except that $q_0$ needs to be replaced everywhere by $\tilde{q}_0$.
    
    \end{proof}

  \medskip
  
  \begin{cor}[(Transversality of singularities of different order)]

    For almost every selection of the masses $m_1>m_2>\dots>m_n$ and for every positive integer $k$
    it is true that for almost every singular phase point $x_0=(q_0, v_0)\in\mathcal{S}_0$ if
    $x_k=T^k(x_0)$ happens to belong to $\mathcal{S}_0$, then $T^k(\mathcal{S}_0)$ and $\mathcal{S}_0$
    are transversal at $x_k$.
    
  \end{cor}

  \begin{proof}
    The statement immediately follows from the facts that

    \begin{enumerate}
      \item[$(i)$] $C_{\epsilon}(x_0)\subset\mathcal{S}_0$,
      \item[$(ii)$] $\Pi\circ T^k:\; C_{\epsilon}(x_0)\to\partial Q$ is locally onto,
      \item[$(iii)$] $\mathcal{S}_0$ is defined purely in terms of the $q$ coordinates.
    \end{enumerate}

    \end{proof}

  \begin{cor}

    For almost every selection of the masses $m_1>m_2>\dots>m_n$ and for
    almost every singular phase point $x_0=(q_0, v_0)\in\mathcal{S}_0$ the local
    stable manifold $\gamma^s(x_0)$ of $x_0$ is transversal to the singularity
    $\mathcal{S}_0$.

    \end{cor}

  \begin{proof}
    This statement follows from the fact that $\gamma^s(x_0)$ is the $C^1$-uniform limit (locally, near
    $x_0$) in the following way:

    \[
    \gamma^s(x_0)=\lim_{k\to\infty}T^{-k}\left[C_\epsilon(T^k(x_0))\right],
    \]

    the inverse images $T^{-k}\left[C_\epsilon(T^k(x_0))\right]$ project locally onto
    $\partial Q$ at $x_0$, and, finally, the maximum rank property is a $C^1$ open property.
    Hence the projection

    \[
    \Pi:\; \gamma^s(x_0)\to \partial Q
    \]

    is locally onto, and $\mathcal{S}_0$ is defined purely in terms of the $q$ coordinates
    in $\partial Q$.

    \end{proof}

  This completes the proof of our Main Theorem.

\medskip

\bibliographystyle{amsplain}

\end{document}